\documentclass[11pt]{amsart}
\usepackage[utf8]{inputenc}
\usepackage[T2A]{fontenc}
\usepackage{amsthm, amsmath, amssymb}
\usepackage{geometry}
\theoremstyle{plain}
\newtheorem*{theorem}{Theorem}

\theoremstyle{definition}

\newcommand{\tr}{\mathop{\mathrm{tr}}\nolimits}
\newcommand{\diag}{\mathop{\mathrm{diag}}\nolimits}
\newcommand{\Mat}{\mathop{\mathrm{Mat}}\nolimits}

\begin{document}

\title[Anti-Frobenius Algebras and AYBE]{Anti-Frobenius Algebras and Associative Yang-Baxter Equation}
\thanks{The work is partially supported by RFBR (project 11-01-00341-a) and by the Ministry of Education and Science of Russian Federation (project 8214).}
 
\author{A.~I.~Zobnin}

\address{Department of Mechanics and Mathematics, Moscow State University, 
119991, Moscow, Russia}

\email{Alexey.Zobnin@gmail.com}


\begin{abstract}
Associative Yang-Baxter equation arises in different areas of algebra, e.g., when studying double quadratic Poisson brackets, non-abelian quadratic Poisson brackets, or associative algebras with cyclic 2-cocycle (anti-Frobenius algebras). Precisely, faithful representations of anti-Frobenius algebras (up to isomorphism) are in one-to-one correspondence with skew-symmetric solutions of associative Yang-Baxter equation (up to equivalence). 
Following the work of Odesskii, Rubtsov and Sokolov and using computer algebra system Sage, we found some constant skew-symmetric solutions of associative Yang-Baxter equation and construct corresponded non-abelian quadratic Poisson brackets.
\end{abstract}

\maketitle

MSC 16T25, 14H70

\section{Introduction}

Let $V$ be a finite dimensional vector space and $r$ be a linear operator on $V \otimes V$.
We consider skew-symmetric solutions of associative Yang-Baxter equation~\cite{Aguiar, Polischuk}
\begin{equation}
\label{AYBE}
r^{13}r^{12} - r^{12}r^{23} + r^{23}r^{13} = 0, \quad r^{12} = -r^{21},
\end{equation}
where $r^{ij}$ denotes an operator $r$ acting on $i$th and $j$th component of $V \otimes V \otimes V$.
Fix a basis $e_{\alpha}$ in $V$ and let 
$$
 r(e_{\alpha} \otimes e_{\beta}) = r^{\gamma \varepsilon}_{\alpha \beta} e_{\gamma} e_{\varepsilon}.
$$
Then~\eqref{AYBE} can be rewritten:
\begin{gather}
\label{AYBE1}
 r^{\gamma \varepsilon}_{\alpha \beta} = -r^{\varepsilon \gamma}_{\beta \alpha},\\
\label{AYBE2}
 r^{\lambda \sigma}_{\alpha \beta} r^{\mu \nu}_{\sigma \tau} + r^{\mu \sigma}_{\beta \tau} r^{\nu \lambda}_{\sigma \alpha} + r^{\nu \sigma}_{\tau \alpha} r^{\lambda \mu}_{\sigma \beta} = 0.
\end{gather}

Such solutions appear in different areas of algebra, e.g., in describing double quadratic Poisson brackets~\cite{Bergh, ORS2012}, non-abelian quadratic Poisson brackets~\cite{ORS} and anti-Frobenius algebras~\cite{Aguiar, ORS}.
Using the latter correspondence we construct some solutions of~\eqref{AYBE}.

%\section{Quadratic Poisson brackets}
In~\cite{ORS} Odesskii, Rubtsov and Sokolov considered a special class of non-abelian linear and quadratic Poisson brackets related to ODE systems of the form
$$
\frac{dx_{\alpha}}{dt} = F_{\alpha}(x_1, \ldots, x_N),
$$
where $x_i$ are $m \times m$-matrices in independent variables and $F_{\alpha}$ are non-commutative polynomials.
They generalized $m$-dimensional Manakov top (which is itself a generalization of $m$-dimensional Euler top)
%$$
% \frac{dx_1}{dt} = x_1^2 x_2 - x_1 x_2^2, \quad \frac{dx_2}{dt} = 0,
%$$
to the case of arbitrary $N$.
%Their ODE systems are integrable for any $m$.
They used bi-Hamiltonian approach, i.e., they constructed a pair of compatible (in some sense) Poisson brackets.
Poisson brackets in question
\begin{itemize}
\item are $GL_m$-adjoint invariant;
\item send traces of any two matrix polynomials to the trace of some other matrix polynomial.
\end{itemize}
Such brackets form an important class, since the corresponding Hamilton operator can be expressed in terms of left and right multiplication operators given by polynomials in matrices $x_1, \ldots, x_N$. They were called \emph{non-abelian Poisson brackets} in~\cite{ORS}.

%It turns out that linear non-abelian Poisson brackets are of the form
%$$
% \left\{x^{j_1}_{i_1, \alpha}, x^{j_2}_{i_2, \beta} \right\} = b^{\gamma}_{\alpha \beta} x^{j_2}_{i_1, \gamma} \delta^{j_2}{i_1} - b^{\gamma}_{\beta, \alpha} x^{j_1}_{i_2, \gamma} \delta^{j_2}_{i_1}.
%$$
%where $x^{j}_{i, \alpha}$ are entries of the $m \times m$-matrix $x_{\alpha}$, Greek indices vary from $1$ to $N$,
%and
%$$
% b^{\mu}_{\alpha \beta} b^{\sigma}_{\mu \eta} = b^{\sigma}_{\alpha \mu} b^{\mu}_{\beta \gamma}
%$$
%The last formula means that $b^{\sigma}_{\alpha \beta}$ are structural constants of an associative algebra.

Quadratic non-abelian Poisson brackets are of the form
$$
 \left\{x^{j_1}_{i_1, \alpha}, x^{j_2}_{i_2, \beta} \right\} = r^{\gamma \varepsilon}_{\alpha \beta} x^{j_2}_{i_1, \gamma} x^{j_1}_{i_2, \varepsilon}
 + a^{\gamma \varepsilon}_{\alpha \beta} x^{k}_{i_1, \gamma} x^{j_2}_{k, \varepsilon} \delta^{j_1}_{i_2} - a^{\gamma \varepsilon}_{\beta \alpha} x^k_{i_2, \gamma} x^{j_1}_{k, \varepsilon} \delta^{j_2}_{i_1}.
$$
We consider the case when $a^{\gamma \varepsilon}_{\alpha \beta}=0$.
Then the constraints on coefficients $r^{\gamma \varepsilon}_{\alpha \beta}$ are precisely~\eqref{AYBE1} and~\eqref{AYBE2}, i.e., tensor $r$ is a constant skew-symmetric solution of the associative Yang-Baxter equation
for an $N$-dimensional vector space~$V$.

The classification of non-abelian quadratic Poisson brackets (even with $a^{\gamma \varepsilon}_{\alpha \beta} = 0$) is an open question.
The complete solution for zero $a$ is known in the case $N=2$ (Aguiar~\cite{Aguiar} and Odesskii, Rubtsov and Sokolov~\cite{ORS2012})
and $N=3$~(Sokolov~\cite{Sokolov2012}). For example, all solutions (up to equivalence) in the case $N = 2$ are either of the form
$$
 r^{21}_{22} = -r^{12}_{22} = \lambda,
$$
or
$$
 r^{22}_{21} = -r^{22}_{12} = \lambda.
$$
(Here we presented only non-zero components of tensor $r$.)
%These solutions correspond to anti-Frobenius algebras $\mathcal{A}_{2,1}$ and $\mathcal{A}^{T}_{2,1}$.

\bigskip

\section{Anti-Frobenius algebras}
It would be tempting to find an appropriate algebraic structure for $r$. 
Odesskii, Rubtsov and Sokolov proved (\cite{ORS}, see also~\cite[Proposition 2.7]{Aguiar}) that the solutions of these equations (up to equivalence, i.e., change of basis) are in one-to-one correspondence with faithful representations of anti-Frobenius algebras (up to isomorphism). 
An associative algebra $\mathcal{A}$ is called \emph{anti-Frobenius} if it is equipped with a non-degenerate anti-symmetric bilinear form $(\cdot, \cdot)$ such that
\begin{equation}
\label{2-cocycle}
 (x, yz) + (y, zx) + (z, xy) = 0
\end{equation}
for all $x, y, z \in \mathcal{A}$. Such form is a cyclic 2-cocycle in the sense of Connes~\cite{Connes}.
This correspondence is constructive, i.e., it is possible to obtain explicitly the components of $r$ from an anti-Frobenius algebra and vice versa.
Precisely, let $\varphi: \mathcal{A} \to \Mat_{N}$ be a faithful representation, $\{e_{\alpha}\}$ be a basis of $\varphi(\mathcal{A})$ and $G$ be the matrix of the bilinear form $(x,y)$.
Then
\begin{equation}
\label{iso}
 r^{ab}_{cd} = g^{\alpha, \beta} e^a_{c, \alpha} e^b_{d, \beta},
\end{equation}
where $(g^{\alpha, \beta}) = G^{-1}$.

Odesskii, Rubtsov and Sokolov~\cite{ORS} considered the following example of an anti-Frobenius algebra:
$\mathcal{A}$ consists of $N \times N$-matrices with zero $N$th row,
and $(x,y) = l([x,y])$ for a generic element $l \in \mathcal{A}^*$. 
(Similar construction for Lie algebras was considered in~\cite{Elashvili}.)
They found the corresponding solution of an associative Yang-Baxter equation,
which is equivalent to the following one:
\begin{equation}
\label{rN1}
 r^{\alpha \beta}_{\alpha \beta} = r^{\beta \alpha}_{\alpha \beta} = r^{\alpha \alpha}_{\beta \alpha} = -r^{\alpha \alpha}_{\alpha \beta} = \frac{1}{\lambda_{\alpha} - \lambda_{\beta}}, \quad \alpha \ne \beta, \quad \alpha, \beta = 1, \ldots, N
\end{equation}
(here $\lambda_1, \ldots, \lambda_N$ are arbitrary pairwise distinct parameters, and the other components of $r$ are zero). We found that this solution can be directly obtained from the anti-Frobenius algebra
$$
 \mathcal{A}_{N,1} = \{ A \in \Mat_N \; \mid \sum_{i} a_{ij} = 0 \quad \forall \, j = 1,\ldots,N \}
$$
equipped with bilinear form
\begin{equation}
\label{form}
 (x, y) = \tr \left([x,y] \cdot \diag(\lambda_1, \ldots, \lambda_N) \right),
\end{equation}
which is isomorphic to $\mathcal{A}$.

\bigskip

\section{Main result}
Let's generalize the construction of $\mathcal{A}_{N,1}$.
Let $M$ be a proper divisor of $N$. We consider $N(N-M)$-dimensional algebra
$$
 \mathcal{A}_{N,M} = \{ A \in \Mat_N \; \mid \sum_{i \equiv r \!\!\pmod{M}} a_{ij} = 0 \quad \forall \, r = 1, \ldots, M, \; \forall \, j = 1,\ldots,N \}
$$
equipped with bilinear form~\eqref{form}.
It is easy to check that this skew-symmetric form satisfies~\eqref{2-cocycle}.

\begin{theorem}
\label{main}
Let $\lambda_i$ be equal to $\lambda_j$ iff $\left[\frac{i}{M}\right] = \left[\frac{j}{M}\right]$.
Then the form~\eqref{form} is non-degenerate, i.e., the algebra $\mathcal{A}_{N,M}$ is anti-Frobenius.
\end{theorem}

\begin{proof}
Let's consider the following basis in $\mathcal{A}_{N,M}$:
$$
 B = \left\{ e_{i,j} := E_{i,j} - E_{\bar{i}_j, j} \; \mid \; \bar{i}_j \ne i \right\}.
$$
Here $\bar{i}_j$ denotes the integer that has the same remainder modulo $M$ as $i$
and the same quotient modulo $M$ as $j$, i.e., if $i = q_i M + r_i$ and $j = q_j M + r_j$ then
$\bar{i}_j = M q_j + r_i$. There are $N(N-M)$ elements in $B$.
Note that 
$$
 e_{i,j} \in B \; \iff \; q_i \ne q_j \; \iff \; e_{j,i} \in B
$$
and $e_{i,i} \notin B$.
By assumption we have $\lambda_j = \lambda_{q_j}$ and hence $\lambda_j = \lambda_{\bar{i}_j}$ for all~$j$.
Let's divide all elements in $B$ into pairs $\{e_{i,j}, \, e_{j,i}\}$.
We have
\begin{gather*}
(e_{i,j}, \, e_{j,i}) =
\tr \left([E_{i,j} - E_{\bar{i}_j, j}, E_{j, i} - E_{\bar{j}_i, i}] \cdot \diag(\lambda_1, \ldots, \lambda_N )\right) = \\
= \lambda_i - \lambda_j
\end{gather*}
and, if $p \ne i$ and $q \ne j$,
\begin{gather*}
(e_{i,j}, e_{q,p}) =
\tr \left([E_{i,j} - E_{\bar{i}_j, j}, E_{q, p} - E_{\bar{q}_p, p}] \cdot \diag(\lambda_1, \ldots, \lambda_N )\right) = \\
= \delta_{i,p}\delta_{j,q}(\lambda_i - \lambda_j) 
- \delta_{i,p}\delta_{j,\bar{q}_p}(\lambda_p - \lambda_{\bar{q}_p}) - \\
- \delta_{\bar{i}_j, p}\delta_{j,q}(\lambda_{\bar{i}_j} - \lambda_j) 
+ \delta_{\bar{i}_j, p}\delta_{j, \bar{q}_p}(\lambda_{\bar{i}_j} - \lambda_j) 
= 0.
\end{gather*}
Thus, $(x,y)$ has a canonical block diagonal form in the basis $B$ with nonzero blocks and is non-degenerate.
\end{proof}

Using~\eqref{iso}, we obtain the following components for the corresponding tensor~$r$:
$$
 r^{ab}_{cd} = \begin{cases}
 \sum\limits_{i, j} \dfrac{1}{\lambda_j - \lambda_i} \left(E_{i,j} - E_{\bar{i}_j, j}\right)^a_c \left(E_{j,i} - E_{\bar{j}_i, i}\right)^b_d
 = \dfrac{(\delta^a_d - \delta^a_{\bar{d}_c}) (\delta^b_c - \delta^b_{\bar{c}_d})}{\lambda_c - \lambda_d}, \text{ if } c \ne d,\\
 0, \text{ otherwise}.
\end{cases}
$$
In particular, $r^{ab}_{cd} = 0$ when either $a \ne d \pmod{M}$ or $b \ne c \pmod{M}$.

Note that for $M = 1$ this formula becomes
$$
 r^{ab}_{cd} = \begin{cases}
 \dfrac{(\delta^a_d - \delta^a_c) (\delta^b_c - \delta^b_d)}{\lambda_c - \lambda_d}, \text{ if } c \ne d,\\
 0, \text{ otherwise},
 \end{cases}
$$
which coincides with~\eqref{rN1}.
As we have seen, the special case $M = 1$ is equivalent to the construction of Odesskii, Rubtsov and Sokolov.

Consider now the case of arbitrary pairwise distinct $\lambda_i$.
One can check that in this case the form $(x, y)$ is non-degenerate too.
Using computer algebra system Sage, we obtained for small $N$ and generalized for all $N$ the following formula for components of tensor $r$ in this case:
\begin{gather*}
r^{a b}_{c d} = 0, \text{ if } a \not \equiv d \text{ or } b \not \equiv c, \\
r^{a a}_{e a} = -r^{a a}_{a e} = \frac{1}{\lambda_a - \lambda_e}, \text{ when } a \ne e, \\
r^{a a}_{c d} = 0, \text{ if } c \ne a \text{ or } d \ne a,\\
r^{a b}_{b a} = 
\frac{1}{\lambda_a - \lambda_b}
\left(
\frac{
\prod\limits_{b' \equiv b, \, b' \ne b} (\lambda_a - \lambda_{b'}) \prod\limits_{a' \equiv a, \, a' \ne a} (\lambda_b - \lambda_{a'})}
{\prod\limits_{a' \equiv a, \, a' \ne a} (\lambda_a - \lambda_{a'}) \prod\limits_{b' \equiv b, \, b' \ne b} (\lambda_b - \lambda_{b'})}
- 1\right), \text{ if } a \ne b, \\
r^{a b}_{c d} = 
\frac{1}{\lambda_a - \lambda_b} \cdot \frac{
\prod\limits_{c' \equiv c, \, c' \ne c} (\lambda_a - \lambda_{c'}) \prod\limits_{d' \equiv d, \, d' \ne d} (\lambda_b - \lambda_{d'})}
{\prod\limits_{a' \equiv a, \, a' \ne a} (\lambda_a - \lambda_{a'}) \prod\limits_{b' \equiv b, \, b' \ne b} (\lambda_b - \lambda_{b'})} \text{ otherwise}.
\end{gather*}
Here $x \equiv y$ means that $x \equiv y \! \pmod{M}$.

With these formulae one can construct corresponding quadratic Poisson brackets.
For example, in the case $N = 2M$ and $m=1$ the corresponding scalar Poisson bracket has the form
$$
 \left\{x_{\alpha}, x_{\beta}\right\} = \frac{(x_{\alpha} - x_{\alpha'})(x_{\beta}-x_{\beta'})(\lambda_{\alpha'}-\lambda_{\beta'})}{(\lambda_{\alpha} - \lambda_{\beta'})(\lambda_{\beta} - \lambda_{\beta'})}
$$
where $\gamma'$ relates to $\gamma$ as $|\gamma' - \gamma| = M$.
%and non-abelian systems of ODEs using the approach of~\cite{ORS}.
%It turned out that only the case $M=1$ considered in~\cite{ORS} leads to a new integrable system of ODEs.

\medskip

Finally, we note that for the similar algebra 
$$
 \mathcal{A}^T_{N,M} = \{ A^T \; \mid \; A \in \mathcal{A}_{N,M}\}
$$
the corresponding solution $\hat{r}$ can be computed as follows:
$$
 \hat{r}^{a b}_{c d} = r^{c d}_{a b}.
$$

\bigskip

\section*{Acknowledgements}
The author is grateful to V. Sokolov for the support and extremely useful discussions.
%The work is partially supported by RFBR (project 11-01-00341-a)
%and by the Ministry of Education and Science of Russian Federation (project 8214).


\bigskip


\begin{thebibliography}{5}\footnotesize
\bibitem{Aguiar}
M.~Aguiar.
\emph{On the associative analog of Lie bialgebras}.
J. Alg., {\bf 244} (2001), no. 2, 429--532.

\bibitem{Bergh}
M.~Van~den~Bergh.
\emph{Double Poisson algebras}.
Trans. Amer. Math. Soc., {\bf 360} (2008), no. 11, 5711--5769.

\bibitem{Connes}
A.~Connes.
\emph{Non-commutative Algebraic Geometry}.
Publ. Math. Inst. Haute \'Etudes Sci. {\bf 62} (1985), 257--360.

\bibitem{Elashvili}
A.~G.~Elashvili.
\emph{Frobenius Lie algebras}. (Russian) Funktsional. Anal. i Prilozhen. {\bf 16}
(1982), no. 4, 9495. English translation: Functional Anal. Appl. {\bf 16} (1982), no. 4, 326--328
(1983).

\bibitem{ORS}
A.~V.~Odesskii, V.~N.~Rubtsov, V.~V.~Sokolov.
\emph{Bi-Hamiltonian ordinary differential equations with matrix variables}.
Theoretical and Mathematical Physics, April 2012, Volume 171, Issue 1, pp.~442--447.

\bibitem{ORS2012}
A.~V.~Odesskii, V.~N.~Rubtsov, V.~V.~Sokolov.
\emph{Double Poisson brackets on free associative algebras}.
Contemporary Mathematics, {\bf 592} (2013), 225--241.

\bibitem{Polischuk}
A. Polischuk.
\emph{Classical Yang–Baxter Equation and the $A_{\infty}$-Constraint}.
Advances in Mathematics, {\bf 168}, issue 1, 56--95.

\bibitem{Sokolov2012}
V.~Sokolov.
\emph{Classification of constant solutions for associative Yang-Baxter equation on $\mathfrak{gl}(3)$}.
Submitted to Theor. Math. Phys., 2013;
arXiv:1212.6421.

\end{thebibliography}
\end{document}